\documentclass[11pt]{article}
\usepackage[margin=1in]{geometry}

\usepackage{mathtools,amssymb,amsthm,xifthen}
\usepackage{amsmath}
\usepackage{graphicx,color}
\usepackage{dsfont}
\usepackage{booktabs}
\usepackage[square,numbers,sort]{natbib}
\usepackage[boxed]{algorithm2e}
\usepackage[hyphens]{url}

\usepackage{hyperref}
\usepackage[capitalize]{cleveref}
\usepackage[notcite,notref,final]{showkeys}
\newtheorem{theorem}{Theorem}
\newtheorem{corollary}{Corollary}
\newtheorem{proposition}{Proposition}
\newtheorem{lemma}{Lemma}
\theoremstyle{definition}
\newtheorem{definition}{Definition}

\newcommand{\cD}{\mathcal{D}}

\newcommand{\RR}{\mathbb{R}}

\newcommand*{\tv}[2]{\mathrm{d_{TV}}(#1, #2)}

\newcommand*{\kl}[3][]{%
\ifthenelse{\isempty{#1}}{\operatorname{D}(#2\,\|\,#3)}%
{\operatorname{D}(#2\,\|\,#3\mid#1)}%
}

\newcommand*{\mi}[3][]{%
\ifthenelse{\isempty{#1}}{\operatorname{I}(#2\,;\,#3)}%
{\operatorname{I}(#2\,;\,#3\mid#1)}%
}

\newcommand*{\ent}[2][]{%
\ifthenelse{\isempty{#1}}{\operatorname{H}(#2)}%
{\operatorname{H}(#2\mid#1)}%
}

\newcommand*{\E}{\mathbb E}

\newcommand*{\ep}{\varepsilon}
\newcommand*{\defeq}{\triangleq}

\newcommand{\inner}[1]{\langle #1 \rangle}

\newcommand*{\mmse}[2]{\operatorname{MMSE}_{#1}(#2)}

\newtheorem{assumption}{Assumption}

\DeclarePairedDelimiter\ceil{\lceil}{\rceil}
\DeclarePairedDelimiter\floor{\lfloor}{\rfloor}

\begin{document}

\title{It was ``all'' for ``nothing'':\\ sharp phase transitions for noiseless discrete channels}



\author{
Jonathan Niles-Weed\thanks{Courant Institute of Mathematical Sciences and Center for Data Science, New York University e-mail: {\tt jnw@cims.nyu.edu}. JNW is supported in part by NSF grant DMS-201529.}
\and
{ Ilias Zadik}\thanks{Center for Data Science, New York University ; e-mail: {\tt zadik@nyu.edu.} IZ is supported by a CDS Moore-Sloan postdoctoral fellowship.}
}

\maketitle

\begin{abstract}%
We establish a phase transition known as the ``all-or-nothing'' phenomenon for noiseless discrete channels.
This class of models includes the Bernoulli group testing model and the planted Gaussian perceptron model.
Previously, the existence of the all-or-nothing phenomenon for such models was only known in a limited range of parameters.
Our work extends the results to all signals with arbitrary sublinear sparsity.

Over the past several years, the all-or-nothing phenomenon has been established in various models as an outcome of two seemingly disjoint results: one positive result establishing the ``all" half of all-or-nothing, and one impossibility result establishing the ``nothing" half. Our main technique in the present work is to show that for noiseless discrete channels, the ``all'' half implies the ``nothing'' half, that is, a proof of ``all'' can be turned into a proof of ``nothing.''
Since the ``all'' half can often be proven by straightforward means---for instance, by the first-moment method---our equivalence gives a powerful and general approach towards establishing the existence of this phenomenon in other contexts.\footnote{A 2-pages extended abstract for this work has been presented at the Conference on Learning Theory (COLT) 2021.}
\end{abstract}

\tableofcontents 
\section{Introduction}
A surprising feature of high-dimensional inference problems is the presence of \emph{phase transitions}, where the behavior of estimator changes abruptly as the parameters of a problem vary.
Often, these transitions help illuminate fundamental limitations of an optimal estimation procedure, by showing, for instance, that a certain inference task is impossible when the noise is too large or the number of samples too few.
There is a large and growing literature on proving rigorously the presence of such transitions and on establishing their implications for learning and inference tasks in a variety of settings.~\citep[see, e.g.,][]{MezMon09} 

A particularly striking phase transition is known as the \emph{all-or-nothing} phenomenon~\citep{gamarnikzadik, ReevesPhenom, Zad19}.
In problems evincing this phenomenon, there is a sharp break: below a critical number of samples, it is impossible to infer \emph{almost any} information about a parameter of interest, but as soon as that critical point is reached, it is possible to infer the parameter almost perfectly.
Such problems exhibit a sharp dichotomy, where either perfect inference is possible or nothing is.

In this work, we develop general tools for proving the all-or-nothing phenomenon for a class of models we call ``noiseless discrete channels."
In such models, we fix a function $g$ and observe identically distributed copies of a pair $(Y, X) \in \mathcal Y \times \RR^{L}$ generated by
\begin{equation*}
Y = g(X, \theta)\,,
\end{equation*}
where $X$ is a random draw from some known distribution on $\RR^L$, and $\theta$ is an unknown parameter to be estimated.
Under the assumption that $|\mathcal Y| < \infty$, we can view $g$ as a discrete channel, parametrized by $\theta$, which maps $\RR^L$ to $\mathcal Y$, and our goal is to ascertain how many samples (i.e., how many uses of this channel) we need to reliably recover $\theta$.

We highlight two special cases of the above model which have seen recent attention:
\begin{itemize}
\item Group testing~\citep{Dor43, survey}: $\theta \in \{0, 1\}^N$ indicates a subset of infected individuals in a population, and $X \in \{0, 1\}^N$ indicates a random subset chosen to be tested as a batch.
We observe $g(X, \theta) = 1( \mathrm{Support}(X) \cap \mathrm{Support}(\theta) \neq \emptyset)$, where for a vector $v \in \mathbb{R}^N$, $\mathrm{Support}(v) \subseteq [N]$ denotes the set of the non-zero coordinates of $v$. 
How many tests do we need to determine which individuals are infected?
\item Planted Gaussian perceptron~\citep{ZdeKrz16}: in this simple ``teacher-student'' setting, $\theta \in \{0, 1\}^N$ represents the weights of a ``teacher'' one-layer neural network, and we observe $g(X, \theta) = 1( \sum_{j=1}^N \theta_j x_j \geq 0)$, where the $x_j$ are i.i.d.~standard Gaussian random variables.
How many samples do we need for a ``student'' to learn the teacher's hidden weights?
\end{itemize}
Both models have recently been studied in the all-or-nothing framework~\citep{ScarlletAll,luneau20}.
However, the range of parameters for which the all-or-nothing phenomenon has been rigorously established in either model is limited.
\cite{ScarlletAll} show that all-or-nothing holds for group testing in the extremely sparse regime when the number of infected individuals is $o(N^\ep)$ for all $\ep > 0$.
Their proof is combinatorial and proceeds by the second-moment method.
\cite{luneau20} give a heuristic derivation of the all-or-nothing phenomenon for the planted Gaussian perceptron based on the replica method from statistical physics, and establish that this phenomenon holds if $\|\theta\|_0:=|\{i \in [N]: \theta_i \not =0 \}|$ is both $\omega(N^{\frac{8}{9}})$ and $o(N)$.

\subsection{Contribution}
We give a simple criterion for the all-or-nothing phenomenon to hold in noiseless discrete channels. For such settings, we measure success in terms of the minimum mean squared error (MMSE) and the signal is assumed to lie on the Euclidean unit sphere. The \textit{``all" property} corresponds to a vanishing MMSE, while the \textit{``nothing" property} corresponds to MMSE being asymptotically equal to one, which is the mean squared error achieved by the trivial zero estimator.
As a corollary of our result, we show that the all-or-nothing phenomenon holds for all relevant sparsity regimes in both the group testing and planted perceptron models, substantially generalizing prior work.

Our key technical contribution is to show that, under suitable conditions, proving the ``all'' condition \emph{immediately} implies that the ``nothing'' condition holds as well.
More specifically, we show that \textcolor{black}{if the mean squared error vanishes} for all $n \geq (1+\epsilon) n^*$ for some critical $n^*$, then for $n \leq (1-\epsilon)n^*$ no recovery is possible.
In other words, for these models, ``all'' implies ``nothing'' in a suitable sense.
Crucially, the ``all'' condition can often be proven directly, by simple means, as it suffices to establish that a specific estimator is successful, via for example a simple ``union bound" or ``first-moment" argument. On the other hand, the ``nothing'' lower bound requires proving the failure of any estimation method, and has typically been proven by using more subtle techniques, such as delicate second moment method arguments (see e.g. \cite{ReevesPhenom} for the regression setting and \cite{ScarlletAll} for the Bernoulli group testing setting).
Our ``all'' implies ``nothing'' result shows that this complication is unnecessary for a class of noiseless discrete channels.


We apply our techniques to both non-adaptive Bernoulli group testing and the planted Gaussian perceptron model. We report the following.
\begin{itemize}
\item For the \textit{Bernoulli group testing model} (BGT), we focus on the case, common in the group testing literature, where there are $k$ infected individuals, with $k=o(N).$ We model the infected individuals as a binary $k$-sparse vector on the unit sphere, and as mentioned above we measure success in terms of the MMSE. In the BGT setting each individual is assumed to participate in any given test in an i.i.d. fashion, and independently with everything, with probability $\frac{\nu}{k},$ for some $\nu=\nu_k$ satisfying $q=(1-\frac{\nu}{k})^k.$ Here $q \in (0,1)$  is a fixed constant, again as customary in the literature of Bernoulli group testing \citep{survey}.  We show as an application of our technique that the all-or-nothing phenomenon holds for the BGT design \textbf{for all $k=o(N)$ and for any $q \leq \frac{1}{2}$} at the critical number of tests $$n^*_q=k \log \frac{N}{k}/h(q),$$where $h(q)$ denotes the (rescaled) binary entropy at $q$ defined in \eqref{bin_ent}. In words, with less than $n_q$ samples the MMSE is not better than ``random guess", while with more than $n_q$ samples it is almost zero. To the best of our knowledge this result was known before only in the case where $k=o(N^{\ep})$ for all $\ep>0$ and $q=\frac{1}{2}$ \citep{ScarlletAll}.

\item For the \textit{Gaussian perceptron model}, we focus on the case where $\theta$ is a a binary $k$-sparse vector on the unit sphere, \textcolor{black}{with $k = o(N)$}. We study a more general class of noiseless Boolean models than the Gaussian perceptron, where $Y_i=1(\inner{X_i,\theta} \in A)$ for some arbitrary Borel $A \subseteq \mathbb{R}$ with (standard) Gaussian mass equal to $\frac{1}{2}$. Equivalently we consider any Boolean function $f: \mathbb{R} \rightarrow \{-1,1\}$ which is balanced under the standard Gaussian measure, i.e. $\E{f(Z)}=0, Z \sim N(0,1),$ and assume $Y_i=f(\inner{X_i,\theta}).$ Notice that the perceptron model corresponds to the case $A=[0,+\infty)$ and $f(t)=2 1(t>0)-1$, but it includes other interesting models such as the symmetric binary perceptron $A=[-u,u]$ with $u$ the median of $|Z|,Z \sim N(0,1)$ which has recently been studied in the statistical physics literature \citep{AubPerZde19}. We apply our technique in this setting to prove a generic result; \textbf{all such models} exhibit the all-or-nothing phenomenon at the same critical sample size $$n^*=k \log_2 \frac{N}{k}.$$ To the best of our knowledge this sharp phase transition was known before only in the case where $A=[0,+\infty)$ and $k$ is $\omega(N^{\frac{8}{9}})$ and $o(N)$ \cite{luneau20}

\end{itemize}

\subsection{Comparison with previous work}

\paragraph{All-or-Nothing} The all-or-nothing phenomenon has been investigated in a variety of models, and with different techniques~\citep{gamarnikzadik,NilZad20,barbier01, ReevesCAMSAP, barbier2020allornothing, luneau20, ScarlletAll, ReevesPhenom}. More specifically, the phenomenon was initially observed in the context of the maximum likelihood estimator for sparse regression in \cite{gamarnikzadik} and was later established in the context of MMSE for sparse regression \citep{ReevesPhenom, ReevesCAMSAP}, sparse (tensor) PCA \citep{barbier01, NilZad20, barbier2020allornothing}, Bernoulli group testing \citep{ScarlletAll} and generalized linear models \citep{luneau20}.

A common theme of these works is that all-or-nothing behavior can arise when the parameter to be recovered is sparse, with sparsity \emph{sublinear} in the dimensions of the problem. 
Though it is expected that this phenomenon should arise for all sublinear scalings, technical difficulties often restrict the range of applicability of rigorous results. In the present work we circumvent this challenge by showing that a version of the ``all" condition suffices to establish the all-or-nothing phenomenon for the whole sublinear regime. As mentioned above, usually the ``all" result is easier to establish than the ``nothing" result. Leveraging this, we are able to establish the all-or-nothing phase transitions throughout the sublinear sparsity regimes of both the Bernoulli group testing and Gaussian perceptron models, where only partial results have been established before \citep{ScarlletAll, luneau20}.

\paragraph{``All" implies ``Nothing"} As mentioned already, our key technical contribution is showing that the ``all" result suffices to establish the all-or-nothing phenomenon. This potentially counterintuitive result relates to a technique used in information theory known as the \textit{area theorem} \cite{measson2008maxwell, kudekar2017reed, reeves:2016a}.
A heuristic explanation of this connection in the regression context appears in \cite[Section 1.1.]{ReevesPhenom}; however, despite this intuition, the authors of \cite{ReevesPhenom} do not proceed by this route.
To the best of our knowledge, our results are the first to rigorously prove that in certain sparse learning settings, the ``all" result indeed implies the all-or-nothing sparse phase transition.

\paragraph{Bernoulli group testing} Now, we comment on our contribution for the BGT model, as compared to the BGT literature. In the Bernoulli group testing model, it is well-known that for all $k=o(N)$ and $q=(1-\frac{\nu}{k})^k$, it is possible to obtain a vanishing MMSE (``all") with access to  $(1+\epsilon)n^*_q=(1+\epsilon) k \log \frac{N}{k}/h(q)$ tests \citep{survey,ScarlletAll}. Furthermore, it is also known that if $q=1/2$ with less than $(1-\epsilon)n^*_{1/2}$ test it is impossible to achieve an ``all" result \citep[Theorem 3]{SODA2016}. To the best of our knowledge, this (weak) negative result of ``all" being impossible is \textit{not known} when $q \not =\frac{1}{2}$ and one has access to fewer than $(1-\epsilon)n^*_q$ tests, \textcolor{black}{though some relevant discussion appears in \citep{SODA2016}.} Finally, as mentioned above, \cite{ScarlletAll} do establish the strong negative ``nothing" result that if $k=o(N^{\delta})$ for all  $\delta>0$ and $q=\frac{1}{2}$ with less $(1-\epsilon)n^*_{1/2}$ it is impossible to achieve a non-trivial MMSE  \citep{ScarlletAll}.

 In the present work, we show as a corollary of our methods that \textbf{for all} $k=o(N)$ and $q \leq \frac{1}{2}$, ``nothing" holds when the number of tests is fewer than $(1-\epsilon)n^*_q$, substantially improving the literature of impossibility results in Bernoulli group testing. While to the best of our knowledge, the appropriate ``all" result needed for our argument to work is not known for any $q < \frac{1}{2}$ it has been established before when $q=\frac{1}{2}$ \citep[see, e.g.][Lemma 1.3.]{iliopoulos2020group}. \textcolor{black}{Finally, it is worth pointing out that some form of non-trivial information can still be extracted from the Bernoulli group testing instance even in the ``nothing'' regime where the MMSE is trivial. For example, \cite{ScarlletAll} showed that for some values of $k$ it is possible even when $n<(1-\epsilon)n^*_q$ to successfully hypothesis test between the Bernoulli group testing model and a ``pure noise'' model where the tests outcomes are random and independent from everything else (see also the more recent work \cite{GT_COLT22} on the same topic).}

\paragraph{Gaussian perceptron model} For the Gaussian perceptron model, to the best of our knowledge the most relevant result is in \cite{luneau20} where the authors prove the all-or-nothing phenomenon at $n^*=k \log_2 \frac{N}{k}$ samples when  $k$ is $\omega(N^{\frac{8}{9}})$ and $o(N)$. While they characterize the  free energy of the model and therefore provide more precise results than we do, their results apply to a restricted sparsity regime. We do not precisely characterize the limiting free energy, but our much simpler argument shows that the all-or-nothing phenomenon holds \textbf{for all} sparsity levels $k=o(N)$.

\section{Main Results}

\subsection{General framework: noiseless discrete channels}\label{sec:framework}
\paragraph{The family of models}
We define a sequence of observational models we study in this work, indexed by $N \in \mathbb{N}$. Assume that an unknown parameter, or ``signal", $\theta \in \mathbb{R}^N$ is drawn from some uniform prior $P_{\Theta}=(P_{\Theta})_{N}$ supported on a discrete subset $\Theta$ of the unit sphere in $\mathbb{R}^N$. We set $|\Theta|=M=M_N$ and make the following \emph{``non-negativity''} assumption on the overlap between two parameters that for any $\theta,\theta' \in \Theta$ it holds $$\inner{\theta,\theta'} \geq 0.$$
 For some distribution $\cD=\cD_N$ supported on $\mathbb{R}^L,$ where $L=L_N$, we assume that for $n=n_N$ i.i.d.~samples $X_i \sim \mathcal{D}_X, i=1,2,\ldots, n$ we observe $(Y_i,X_i), i=1,2,\ldots,n$ where
\begin{align}\label{eq:observ}
Y_i=g(X_i,\theta), i=1,2,\ldots, n.
\end{align} The function $g=g_N: \mathbb{R}^L \times \mathbb{R}^N \rightarrow \mathcal{Y}$ is referred to as the channel. We assume throughout that $\mathcal{Y}$ is finite and of cardinality that remains constant as $N$ grows, e.g.,~$\mathcal{Y}=\{0,1\}$.  We denote by $Y^n$ the $n$-dimensional vector with entries $Y_i, i=1,2,\ldots, n$ and $X^n$ the $n \times L$ matrix with columns the vectors $X_i, i=1,2,\ldots, n$.
We write $\mathrm P = \mathrm P_N$ for the joint law of $(Y^n, X^n, \theta)$.

We are given access to the pair $(Y^n,X^n)$,and our goal is to recover $\theta$.  We measure recovery with $n$ samples in terms of the minimum mean squared error (MMSE),\begin{align}
\mathrm{MMSE}_N(n)= \E \|\theta-\E [\theta|Y^n,X^n] \|^2\,.
\end{align}

%
%
%
%
%

\paragraph{The all-or-nothing phenomenon}
We say that a sequence of models $((P_{\Theta})_N, g_N, \cD_N)$ satisfies the \emph{all-or-nothing phenomenon} with critical sequence of sample sizes $n_c=(n_c)_N$ if 
\begin{equation}\label{all_or_nothing}
\lim_{N \to \infty} \mmse{N} { \floor{\beta n_c}} = \left\{
\begin{array}{ll}
1 & \text{ if $\beta<1$} \\
0 & \text{ if $\beta>1$}\,.
\end{array}
\right.
\end{equation}
This condition expresses a very sharp phase transition: when $\beta > 1$, we can identify the signal nearly perfectly, but when $\beta < 1$, we can do no better than a trivial estimator which always outputs zero.

\paragraph{Assumptions} To establish our result we make throughout the following further assumptions on our models.

Recall that we have assumed that our prior $P_{\Theta}$ is the uniform distribution on some finite subset of cardinality $M=M_N$.
We assume throughout that $M_N \to \infty$ as $N \to \infty$.
We also make the following assumption, which requires that the distribution $ P_{\Theta}$ is sufficiently spread out. \begin{assumption}\label{assum:full_ov}
For $\theta$ and $\theta'$ chosen independently from $P_{\Theta}$ we have
\begin{align}\label{assum:full_ov2}
\lim_{\delta \rightarrow 0^+} \lim_{N \rightarrow +\infty} \frac{\log ( M P_{\Theta}^{\otimes 2}(\inner{ \theta, \theta' } \geq 1-\delta))}{\log M}=0.
\end{align}
Moreover, we assume for \textcolor{black}{$\theta$ and $\theta'$ chosen independently from $P_{\Theta}$} and any $\epsilon>0$,
\begin{align}\label{in_prob}
\lim_N P_{\Theta}^{\otimes 2}(\inner{\theta',\theta} \geq \epsilon)=0.
\end{align} 
\end{assumption}
Assumption \eqref{assum:full_ov2} guarantees that the that for two independent draws from the prior $\theta,\theta'$, the asymptotic probability that $\theta'$ is very near to $\theta$ is dominated by the probability that $\theta = \theta'$.
This condition is the same as the one employed by~\cite{NilZad20} in the analysis of the all-or-nothing phenomenon for Gaussian models.

Assumption~\eqref{in_prob} implies that \textcolor{black}{independent samples from the prior are asymptotically uncorrelated with each other. This condition is natural in the context of the all-or-nothing phenomenon, since if Assumption~\eqref{in_prob} fails to hold, then it is possible to obtain an estimator with non-trivial correlation with the signal by simply drawing a fresh sample from the prior, independent of the observations.}


\textcolor{black}{Assumptions \eqref{assum:full_ov2} and \eqref{in_prob} are easy to verify in a variety of sparse models. For instance, they hold if the rate function
\begin{equation*}
	r(\rho) = - \lim_N \frac{1}{\log M} \log P_{\Theta}^{\otimes 2}(\langle \theta', \theta \rangle \geq \rho) \quad \quad \rho \in [0,1]
\end{equation*}
exists and is a strictly increasing continuous function on $[0, 1]$.
}

We make also assumptions on the probability a $\theta' \in \Theta \setminus \{\theta\}$ is able to fit the observations generated by the signal $\theta$.
\begin{assumption}\label{assum:r_low_ov}
We assume there exists a fixed function $R:[0, 1] \to [0, 1]$, independent of $N$, such that
\begin{equation*}
P_N(g(X, \theta) = g(X, \theta')) = R(\inner{\theta, \theta'}) \quad \forall N \in \mathbb N, \theta, \theta' \in \Theta\,.
\end{equation*}
That is, that the probability that $g(X, \theta)$ and $g(X, \theta')$ agree is a function of $\inner{\theta, \theta'}$ alone.
We assume that $R$ is continuous at $0^+$ and strictly increasing on $[0, 1]$.
\end{assumption}

\subsection{Main result: How ``all" implies ``nothing"}

Notice that since our prior distribution is a uniform distribution over the finite parameter space $\Theta$ and our observation model is noiseless, the posterior distribution of $\theta$ given $Y^n,X^n$ satisfies that for any $\theta'$,
\begin{align*}
P(\theta'|Y^n,X^n) = \frac{ P(\theta')P(Y^n|X^n,\theta')}{P(Y^n|X^n)} \propto  P(Y^n|X^n,\theta')=\prod_{i=1}^n 1(Y_i=g(X_i,\theta')).
\end{align*} In words, the posterior distribution is simply the uniform measure over the vectors $\theta' \in \Theta$ satisfying 
\begin{align}\label{eq:observ2}
Y_i=g(X_i,\theta'), i=1,2,\ldots, n.
\end{align} As an easy corollary, the distance of the posterior mean from the ground truth vector, or equivalently the $\mmse{N}{n}$, can be naturally related to the behavior of the following ``counting" random variables. 
\begin{definition}\label{dfn:ZN}
 For any $N \in \mathbb{N}$ and $\delta \in [0,2]$, let $Z_{N,\delta}=Z_{N,\delta}(Y^n,X^n)$ be the random variable which is equal to the number of solutions $\theta' \in \Theta$ of equations \eqref{eq:observ2} where $\|\theta-\theta'\|^2 \geq \delta.$
\end{definition}
Using the definition above, the following simple \textcolor{black}{proposition} holds.
\begin{proposition}\label{prop:postCounting}
For $\theta'$ drawn from the posterior distribution of $\theta$ given $Y^n,X^n$ it holds almost surely that
\begin{align}\label{eq:posterior}
P(\|\theta-\theta'\|^2 \geq \delta|Y^n,X^n)=\frac{Z_{N,\delta}}{Z_{N,0}}.
\end{align}Hence,
\begin{align}\label{eq:posteriorExp}
\mmse{N}{n} = \frac 12 \E \int_{\delta=0}^{2}\frac{Z_{N,\delta}}{Z_{N,0}} d\delta.
\end{align} Furthermore, the property that
\begin{align}\label{assum:no_sol_new}
\text{ for all } \delta \in (0,2], \lim_N \E \frac{Z_{N,\delta}}{Z_{N,0}}=0 
\end{align} is equivalent with the ``all" property
\begin{align}\label{mmse0}
\lim_N \mmse{N}{n} =0.
\end{align} Finally the property \begin{align}\label{assum:no_sol}
\text{ for all } \delta \in (0,2], \lim_N P(Z_{N,\delta}>0)=0,
\end{align}implies the ``all" condition \eqref{mmse0}.
\end{proposition}

Proposition \ref{prop:postCounting} offers a clean combinatorial way of establishing the vanishing MMSE (``all") in our context; one needs to prove the (relative) absence of solutions of \eqref{eq:observ2} which are at a constant distance from $\theta$, establishing for example \eqref{assum:no_sol}. A clear benefit of such an approach is that one could possibly establish such a result by trying a (possibly conditional) union bound---or ``first moment"---argument. We investigate further the power of establishing the ``all" result in what follows.

We consider the following \textit{critical sample size}, 
\begin{equation}\label{eq:nstar}
n^* = (n^*)_N =\floor*{\frac{H(\theta)}{H(Y)}},
\end{equation}where by $H(\cdot)$ we refer to the Shannon entropy of a discrete random variable and $Y=g(X,\theta)$ for a sample of $X \sim \cD_N$ and $\theta \sim (P_{\Theta})_N$. The significance of the sample size $n^*$ is highlighted in the following proposition which establishes that the ``all" condition \eqref{mmse0} can only hold if the number of samples is at least $n^*$.
\begin{proposition}\label{prop:lb}
Suppose that Assumption \ref{assum:full_ov} is true. If the ``all" condition \eqref{mmse0} holds for some sequence of sample sizes $n=n_N$, then
\begin{align*}
\liminf_N \frac{n}{n^*} \geq 1.
\end{align*}

\end{proposition}
While we defer the proof of Proposition \ref{prop:lb} to the Appendix \ref{sec:remain}, we highlight some aspects of it which will be important in what follows. The key identity behind the proof of the proposition is it always holds that
\begin{align}
H(\theta)-H(\theta|Y^n,X^n) &=n H(Y)-\kl{P(Y^n,X^n)}{Q (Y^n,X^n)} \label{eq:key_id1_main}\\
& \leq n H(Y),\label{ineq:key_id1}
\end{align}where 1) $D$ stands for the Kullback-Leibler (KL) divergence (see e.g. \cite[Section 6]{PWlt}), 2) $P(Y^n,X^n)$ stands for the joint law of $(Y^n, X^n)$ generated by the observation model \eqref{eq:observ} and 3) $Q(Y^n,X^n)$ stands for the law of a ``null" model where the columns of $X^n$ are  i.i.d. samples drawn from $\mathcal{D}$ and the entries of $Y^n$ are drawn in an i.i.d. fashion from the distribution of $Y=g(X,\theta)$ \textit{but independently from $X^n$}. \textcolor{black}{As a result, the law of a single observation $(X_i, Y_i)$ is the same under $P$ and $Q$,} but \textcolor{black}{$P$ and $Q$} are distinct as joint distributions, as for example the latter has \textit{no hidden signal}. The identity \eqref{eq:key_id1_main} follows from algebraic manipulations which can be found in Appendix~\ref{app:conv}. The inequality in \eqref{ineq:key_id1} is implied by the non-negativity of the KL divergence.

The proof of the proposition is based on the fact that the ``all" condition \eqref{mmse0} implies that the entropy of the posterior is of smaller order of magnitude than the entropy of the prior (see Proposition \ref{prop:ent_post}). This property allows us to conclude that the left hand side of \eqref{ineq:key_id1} is $(1-o(1))H(\theta)$ which concludes the proof.

Now we present the main technical result of the present work. We establish that if Proposition \ref{prop:lb} is tight, that is if \eqref{mmse0} can be proven to be true when $n \geq (1+\epsilon)n^*$ for arbitrary $\epsilon>0$, then \eqref{mmse0} is a sufficient to establish that the \textit{the all-or-nothing phenomenon}  holds at sample size $n^*$ as well.

\begin{theorem}[``all" implies ``nothing"]\label{thm:main}
Suppose that Assumptions \ref{assum:full_ov}, \ref{assum:r_low_ov} are true. Assume that if $n \geq (1+\epsilon)n^*$, for some arbitrary fixed $ \epsilon>0,$ then the ``all" condition \eqref{mmse0} holds.  Then if $n \leq (1-\epsilon) n^*$ for arbitrary fixed $\epsilon \in (0,1),$ it holds
\begin{align}\label{mmse1}
\lim_N \mmse{N}{n} =1.
\end{align} In particular, the all-or-nothing phenomenon \eqref{all_or_nothing} holds at critical samples sizes $n_c=n^*$.
\end{theorem}

We provide here some intuition behind such a potentially surprising implication. Notice that if \eqref{mmse0} holds at sample sizes $(1+\epsilon)n^*$ for arbitrary fixed $\epsilon>0$, then from the sketch of the proof of Proposition \ref{prop:lb} the inequality  \eqref{ineq:key_id1} needs to hold (approximately) with equality. In fact one can show that at $n=n^*$, it must necessarily hold that
\begin{align}\label{kl_cond}
\lim_N   \frac{\kl{P(Y^{n^*},X^{n^*})}{Q (Y^{n^*},X^{n^*})}}{H(\theta)}=0.
\end{align} At an intuitive level, \eqref{kl_cond}  seems already a significant step towards what we desire to prove. Indeed, \eqref{kl_cond} suggests that $(Y^n, X^n)$ drawn from our model $P$ are close in distribution to the samples $(Y^n,X^n)$ drawn from the null model $Q$. This strongly suggests that outperforming the random guess in mean-squared error should be impossible.

While we think that this argument hints at the right direction for proving the ``nothing" property, we do not know a complete proof along these lines. The reason is that one cannot conclude that $P$ and $Q$ are ``sufficiently close", e.g., in the total variation sense, to argue the above. The reason is that the KL distance in \eqref{kl_cond} vanishes only after rescaling by the factor $H(\theta)=\log M_N \rightarrow +\infty$. For this reason,  \eqref{kl_cond}  does not imply any nontrivial bound for the total variation distance between $P,Q$.  Notably though, such an obstacle has already been tackled in the literature of ``nothing" results in the context of sparse tensor PCA~\citep{BanksIT, NilZad20}. In these cases the ``nothing" result can be established by the use of the I-MMSE formula combined with weak detection lower bound such as \eqref{kl_cond}. The I-MMSE formula is an identity for Gaussian channels between the derivative (with respect to the continuous \textit{signal to noise ratio (SNR)} ) of the corresponding KL divergence and the MMSE for this value of SNR. We are not aware of any such  formula for the noiseless discrete models considered in this work. Nevertheless, inspired by the I-MMSE connection,  we study the discrete derivative of the KL divergence for noiseless models. Specifically we prove a result of potentially independent interest, that a vanishing \textit{discrete derivative} (with respect to \textit{the sample size} $n$) of the sequence $ \kl{P(Y^{n},X^{n})}{Q (Y^{n},X^{n})}/H(\theta), n \in \mathbb{N}$, at $n \leq n^*$, implies indeed a trivial MMSE at $n < n^*$.  We conclude then the result by using classical real analysis results, to show that the convexity of the vanishing sequence for $n \leq n^*$, implies its the  discrete derivative of the sequence is also vanishing for $n \leq n^*$.

\subsection{The case of Boolean channels: a simple condition}
One can naturally ask whether for various models of interest there exist a simple sufficient condition which can establish the positive result \eqref{assum:no_sol} at $n^*$ (e.g. by a union bound argument) and therefore prove the all-or-nothing phenomenon.  In this subsection, we provide such a simple sufficient condition for the subclass of \textit{Boolean (or 1-bit) noiseless models}, i.e. when $\mathcal{Y}=\{0,1\}$. Perhaps not surprisingly, our result follows from an appropriate ``union-bound" or ``first-moment" argument.  In the next section we apply our condition to various such models. 

Notice that in these Boolean binary settings the critical sample size simplifies to $n^*=\floor{\frac{H(\theta)}{h(p)}}$ where $p=P(g(X,\theta)=1)$ and $h$ is the binary entropy 
\begin{align}\label{bin_ent}
h(t)=-t \log t -(1-t)\log (1-t), t\in (0,1),
\end{align} where $\log$ is, as always in this work, with base $e$.

To proceed, we need some additional definitions. The first is about the two possible outcomes of the channel, and it extends Assumption \ref{assum:r_low_ov} to further properties on the probability of a fixed $\theta' \in \Theta \setminus \{\theta\}$ satisfying \eqref{eq:observ2}.
\begin{assumption}\label{assum:r_new} There exist fixed functions $R_{1}: [0,1] \rightarrow [0,1] ,R_{0}:  [0,1] \rightarrow [0,1] $, independent of $N,$ such that
for all $N \in \mathbb{N}$ and $\theta, \theta' \in \Theta$ it holds
\begin{align}
&R_1({\inner{\theta,\theta'}}) = P(g(X, \theta) = g(X, \theta')=1)  \label{dfn:rho0},\\
& R_0({\inner{\theta,\theta'}}) = P(g(X, \theta) = g(X, \theta')=0) \label{dfn:rho1}.
\end{align} For all $\rho \in [0,1]$, $R(\rho)=R_0({\rho}) +R_1({\rho}),$ where $R(\rho)$ is as in Assumption~\ref{assum:r_low_ov}.
We assume that both $R_0$ and $R_1$ are increasing on $[0,1]$.
\end{assumption}
The second is about the distribution of the overlap between two independent copies of the prior distribution. Again we borrow the definition from~\cite{NilZad20}.

\begin{definition}\label{dfn:ov_rate}
Given a non-decreasing continuous function $r: [-1, 1] \to \RR_{\geq 0}$, we say $\{P_\Theta\}$ admits an \emph{overlap rate function} $r$, \textcolor{black}{if for all $\rho \in [0,1]$ it holds} 
\begin{equation*}
\limsup_{N } \frac{1}{\log M_N} \log \mathrm P_\Theta^{\otimes 2}[\inner{ \theta',\theta} \geq \rho] \leq - r(\rho)\,,
\end{equation*}
where $\theta$ and $\theta'$ are independent draws from $ P_\Theta$.
\end{definition}

We state our result.

\begin{corollary}\label{cor:main}
Let $\mathcal{Y}=\{0,1\}$ and let $p=P(g(X,\theta)=1) \in (0,1)$ be constant. Suppose that Assumptions \ref{assum:full_ov}, \ref{assum:r_low_ov} and \ref{assum:r_new} are true. If $\{P_\Theta\}$ admits an overlap rate $r(\rho)$ satisfying
\begin{align}\label{cond:main}
r(\rho) \geq \frac{1}{h(p)} \left(p \log \frac{R_1({\rho})}{p^2}+(1-p) \log \frac{R_0({\rho})}{(1-p)^2} \right) \quad \forall \rho \in [0, 1],
\end{align}then the all-or-nothing phenomenon holds at $n^*=\floor{\frac{\log M}{h(p)}}.$

\end{corollary}

\section{Applications}\label{sec:apps}

In this subsection we use our results, and specifically Corollary \ref{cor:main}, to establish the all-or-nothing phenomenon for various \textit{sparse Boolean} models of interest.
\subsection{Application 1: Nonadaptive Bernoulli group testing}

 We start with nonadaptive Bernoulli group testing. In this context, we fix some parameter $k=k_N \in \mathbb{N}$ with $k \rightarrow +\infty$ and $k=o(N),$ which corresponds to $k$ infected individuals out of a population of cardinality $N$. 
We also fix a constant $q \in (0, 1/2]$, which controls the size of each group which is getting tested. 
 The goal is to perform nonadaptive tests for a disease on different ``groups" or subsets of the $N$ individuals at a time. The logic is that by doing so we may be able to use fewer tests, say, from testing each individual separately, and still recover the infected individuals. Notice that such a Bernoulli group testing model is characterized by the two parameters $q$ and $k=k_N$.

\paragraph{The Model} We assume a uniform prior $P_{\Theta}$, which following our notation we encode as the uniform measure on the $k$-sparse binary vectors on the sphere in $N$-dimensions, i.e. $$\Theta=\{\theta \in \{0,\frac{1}{\sqrt{k}}\}^N: \|\theta\|_0=k\},$$where there is a natural correspondence between the identities of the infected individuals and the support of the vectors $\theta \in \Theta$.
For each $k$, we define $\nu = \nu_k$ to be the unique positive number satisfying
\begin{equation*}
\left(1 - \frac \nu k\right)^k = q\,.
\end{equation*}
The group of individuals being tested is modeled by the binary vector $X \in \{0,1\}^N$, with $\mathcal{D}=\mathrm{Bernoulli}(\frac{\nu}{k})^{\otimes N}.$ In words, we choose whether in individual participates at any given test independently from everything and with probability $\nu/k$, with the parameter $\nu$ chosen so that the probability that each group contains no infected individuals is exactly $q$.

We model the channel by the step function at $\frac{1}{\sqrt{k}}$, i.e.
$Y=g(X,\theta)=1(\inner{X_i,\theta} \geq \frac{1}{\sqrt{k}}),$ which simply outputs the information of whether at least one of the individuals in the selected group is infected (which is equivalent to $\sqrt{k}\inner{X_i,\theta} =|\mathrm{Support}(X_i) \cap \mathrm{Support}(\theta)| \geq 1$) or not. The sample size $n$ corresponds to the number of tests conducted.

Corollary \ref{cor:main} when applied to this context establishes the following result.
\begin{theorem}\label{thm:BGT}
Let $q \in (0,\frac{1}{2}]$ be a constant. Suppose $k=o(N)$ and $\nu_k$ satisfies $(1-\frac{\nu_k}{k})^k=q$. Then the non-adaptive Bernoulli group testing model satisfies the all-or-nothing phenomenon at number of tests $n^*=\floor*{\frac{\log \binom{N}{k}}{h(q)}} = (1 +o(1))\frac{k \log \frac{N}{k}}{h(q)}.$

\end{theorem}

The $q = 1/2$, $k = N^{o(1)}$ case of Theorem~\ref{thm:BGT} was proved by \cite{ScarlletAll}.
Our theorem extends their result to all sublinear sparsities and all $q \leq 1/2$. 
\textcolor{black}{In particular, we cover the commonly used choices of $q=1/e-o(1),$ with $\nu=1$, and $q=1/2-o(1)$ with $\nu=\ln 2$ (see e.g. \citep[Section 2]{survey}). When $q > 1/2$, condition~\eqref{cond:main} fails, and whether a result similar to Theorem~\ref{thm:BGT} holds in this regime remains open.}


\subsection{Application 2: Sparse Gaussian perceptron and Sparse Balanced Gaussian (SBG) models}

In this subsection, we turn our study to a family of what we call as \textit{Sparse Balanced Gaussian (SBG) models}.  Every such model can be characterized by some sparsity parameter $k=k_N=o(N)$ and a fixed ``balanced" Borel subset $A \subseteq \mathbb{R}$ with $P(Z \in A)=1/2$ for $Z \sim N(0,1)$.

\paragraph{ The Model} We assume as above that the signal $\theta$ is sampled from the uniform measure on the $k$-sparse binary vectors on the sphere in $N$-dimensions, i.e. $\Theta=\{\theta \in \{0,\frac{1}{\sqrt{k}}\}^N: \|\theta\|_0=k\}.$ We assume that the distribution for $X \in \RR^N$ is given by the standard Gaussian measure  $\mathcal{D}=N(0,I_N).$ Finally the channel is given by the formula $Y=g(X,\theta)=1(\inner{X_i,\theta} \in A).$

We highlight two models of this class that have been studied in different contexts.
\begin{itemize}
\item The case $A=[0,+\infty)$ corresponds to the well-studied \textit{ Gaussian perceptron} model with a sparse planted signal $Y=1(\inner{X,\theta} \geq 0).$ Variants of the Gaussian perceptron model have received enormous attention in learning theory and statistical physics (see e.g. ~\citep{ZdeKrz16,BarKrzMac19} and references therein). Recently the sparse version has been studied by~\cite{luneau20}.
\item The case $A=[-u,u]$ for some $u$ with $u$ such that $P(|Z| \leq u)=\frac{1}{2}$, which corresponds to what is known as the \textit{symmetric binary perceptron} model with a sparse planted signal $Y=1(|\inner{X,\theta} | \leq u)$~\citep{AubPerZde19}.
\end{itemize}
We establish a general result that all SBG models exhibit the all-or-nothing phenomenon at the same critical sample size.
\begin{theorem}\label{thm:SBG}
Suppose $k=o(N)$ and $A \subseteq \mathbb{R}$ be an arbitrary fixed Borel subset with $P(Z \in A)=1/2$ for $Z \sim N(0,1)$. Then the Sparse Balanced Gaussian model defined by $k$ and $A$ exhibits the all-or-nothing phenomenon at $n^*=\floor{ \frac{\log  \binom{N}{k}}{\log 2} }=(1+o(1))k \log _2\frac{N}{k}.$

\end{theorem}

In the context of Gaussian perceptron it has recently been proven~\citep{luneau20} that the all-or-nothing phenomenon holds for any $\omega(N^{\frac{8}{9}})=k=o(N)$. Theorem \ref{thm:SBG} generalizes this result to any $k=o(N)$, even constant.
To our knowledge, the existence of such a transition for the symmetric binary perceptron and other SBG models is new.

\bibliographystyle{alpha}
\bibliography{colt}

\newpage
\appendix

\section{Auxilary results and important preliminary concepts}

\textcolor{black}{All the results in this auxilary section holds under our framework and assumptions as described in Section \ref{sec:framework}.}

We start with an elementary lemma. \begin{lemma}\label{lem:nstar} The sequence of critical sample sizes $n^*$ defined in \eqref{eq:nstar} satisfies,
\begin{itemize}
\item[(i)] $\lim_N n^*=+\infty$
\item[(ii)] $n^*=(1+o(1))\frac{H(\theta)}{H(Y)},$ as $N \rightarrow +\infty$.

\end{itemize}

\end{lemma}

\begin{proof}
Recall that in all our models we assume $H(\theta)=\log M_N \rightarrow +\infty$ as $N \rightarrow +\infty$. Furthermore, we assume that $Y=g(X,\theta)$ is a random variable supported on a subset of $\mathcal{Y}$ where $|\mathcal{Y}|=O(1)$. Hence, $H(Y) \leq \log |\mathcal{Y}|=O(1).$ In particular in all our models it holds $$\lim_N \frac{H(\theta)}{H(Y)}=+\infty.$$ This establishes the first property. The second property follows since for any real valued sequence $X_N, N \in \mathbb{N}$ with $\lim_N x_N =+\infty,$ it holds $ \lim_N \frac{x_N}{\floor{x_N}}=1.$

\end{proof}

We continue with a crucial proposition for our main result. This proposition establishes a connection between the critical sample size $n^*$, the estimation error manifested in the form of the entropy of the posterior distribution.

To properly establish it we need some additional notation, and the definition of an appropriate ``null" distribution on the observation $(Y^n,X^n)$. 

\begin{definition} \label{dfn:planted} 
We denote by $P_n=P_n(Y^n,X^n)$ the law of the observation under our model, i.e. for any measurable $A$, $P_n(A)=\E_{\theta \sim P_{\Theta}} P( (Y^n,X^n, \theta) \in A).$ In words, $P_n$ generates $(Y^n,X^n)$ by first sampling $\theta$ from the prior, independently sampling $X^n$ in an i.i.d. fashion from $\cD$, and then generating $Y^n$ by the conditional law $P(Y^n|\theta, X^n)$.
\end{definition} 
Notice that in the noiseless case studied in this work, the latter conditional law greatly simplifies to a dirac mass at $Y^n=(g(X_i,\theta))_{i=1}^n.$

\begin{definition}\label{dfn:null} 
Denote, as usual, by $Y=g(X,\theta)$ the random variable where $X \sim \cD$ and $\theta \sim P_{\Theta}$ are independent. We define by $Q_n=Q_n(Y^n,X^n)$ the ``null" distribution on $n$ samples, where the observations are generated as follows. We sample $X^n$ in an i.i.d. fashion from $\cD$ and then generate $Y^n$ in an i.i.d. fashion from the law of $Y$, \textit{independently} from $X^n$.
\end{definition} Notice first that the marginals of $Y_i$ under $Q_n$ are identical to the marginals of $Y_i$ under $P_n$. Yet, the joint law of $Q_n$ does not include any ``signal" $\theta$ and $Y^n$ are independent of $X^n$. Naturally, is not possible to estimate any signal $\theta$ with observations coming from the null model $Q_n$.

The following proposition holds.
\begin{proposition}\label{prop:key_identity}
For $(\theta,Y^n,X^n)$ generated according to $P_n$, and $Q_n$ defined in Definition \ref{dfn:null} we have,
\begin{align}\label{eq:key_id1}
\mi[X^n]{\theta}{Y^n} =n H(Y)-\kl{P_{n}}{Q_n}.
\end{align} and therefore
\begin{align}\label{eq:key_id2}
\left(1-\frac{H(\theta|Y^n,X^n)}{H(\theta)}\right)+ \frac{\kl{P_{n}}{Q_n}}{H(\theta)}=(1+o(1))\frac{n}{n^*}.
\end{align} 

\end{proposition}

\begin{proof}
Note that \eqref{eq:key_id2} follows from \eqref{eq:key_id1} directly from part (ii) of Lemma \ref{lem:nstar}, along with the identity
\begin{align*} 
\mi[X^n]{\theta}{Y^n}=H(\theta|X^n)-H(\theta|Y^n,X^n)=H(\theta)-H(\theta|Y^n,X^n),
\end{align*}
where the second equality uses that $X^n$ is independent of $\theta$.

We now prove \eqref{eq:key_id2}. We have
\begin{align*}
\mi[X^n]{\theta}{Y^n} & = \E_{X^n} \left[\E_{\theta, Y^n \mid X^n} \log \frac{\mathrm P(Y^n \mid X^n, \theta)}{\mathrm P(Y^n \mid X^n)}\right] \\
& =  \E_{X^n} \left[\E_{\theta, Y^n \mid X^n} \log \frac{1}{\mathrm P(Y^n \mid X^n)}\right]\\
& = \E_{X^n}\left[\E_{\theta, Y^n \mid X^n} \log \frac{1}{\mathrm Q(Y^n)}\right]  -  \E_{X^n}\left[\E_{\theta, Y^n \mid X^n} \log \frac{\mathrm P(Y^n \mid X^n)}{\mathrm Q(Y^n)}\right]
\end{align*}
The second term is $\kl{P_{n}}{Q_n}$.
For the first term, we have that $Q(Y^n) = \prod_{i=1}^n Q(Y_i)$ by the definition of $Q$. Since by assumption also $Q(Y_i)=P(Y_i)$ for each $i$ we conclude 
$$ \E_{X^n}\left[\E_{\theta, Y^n \mid X^n} \log \frac{1}{\mathrm Q(Y^n)}\right]=n \E_{X_1}\left[\E_{\theta, Y_1 \mid X_1} \log \frac{1}{\mathrm P(Y_1)}\right]=n H(Y).$$
The proof is complete.

\end{proof}

\begin{proposition}\label{prop:ent_post}

 Suppose that  \eqref{mmse0} holds for some sequence of sample size $n=n_N$.
Then we also have
\begin{align}\label{eq:all_fake}
\lim_{N \rightarrow +\infty} \frac{H(\theta|Y^n,X^n)}{H(\theta)}=0.
\end{align}
\end{proposition}

\begin{proof}

Recall that since the prior is uniform and the model is noiseless, the posterior is simply the uniform distribution over the solutions $\theta'$ of the system of equations \eqref{eq:observ2}. Hence, using the notation of Definition \ref{dfn:ZN}, $Z_{N,0}$ is the random variable which is equal to the number of such solutions. Hence, \eqref{eq:all_fake} is equivalent to
\begin{align*}
\lim_{N \rightarrow +\infty} \frac{\E \log Z_{N,0} }{\log M}=0.
\end{align*}where we used that $H(\theta)=\log M$.

Now fix a $\delta \in (0,2]$ and let $A_{\delta}:=\{\theta'' \in \Theta: \|\theta''-\theta\| \geq \delta\}.$ Notice that almost surely \begin{align*} \frac{Z_{N,\delta}}{Z_{N,0}} Z_{N,0}=Z_{N,\delta} \leq |A_{\delta}|\end{align*} and 
\begin{align*}
(1-\frac{Z_{N,\delta}}{Z_{N,0}} )Z_{N,0}=Z_{N,0}-Z_{N,\delta} \leq |\Theta \setminus A_{\delta}|.
\end{align*} Hence if we denote for simplicity $p_{\delta}:= \frac{Z_{N,\delta}}{Z_{N,0}}$ we have,
\begin{align*}
\log Z_{N,0} &  \leq p_{\delta} \log ( \frac{|A_{\delta}|}{p_{\delta}})+(1-p_{\delta}) \log ( \frac{|\Theta \setminus A_{\delta}|}{1-p_{\delta}})\\
&  \leq p_{\delta} \log M+\log |\Theta \setminus A_{\delta}| +h(p_{\delta}),
\end{align*}where $h$ is the binary entropy. Since $h(p_{\delta}) \leq \ln 2$ and $ \lim_N \log M_N=+\infty$, we conclude that for any $\delta \in (0,2]$
\begin{align*}
\frac{\E \log Z_{N,0}}{ \log M} & \leq \E p_{\delta}+\E \frac{\log |\Theta \setminus A_{\delta}|}{\log M}+\frac{\ln 2}{\log M}\\
&= \E \frac{Z_{N,\delta}}{Z_{N,0}}+ \E \frac{ \log ( M P_{\theta' \sim P_{\Theta}}(  \|\theta-\theta'\| < \delta))}{\log M}+\frac{\ln 2}{\log M}\\ 
&= \E \frac{Z_{N,\delta}}{Z_{N,0}}+ \E \frac{ \log ( M P_{\theta' \sim P_{\Theta}}(  \inner{ \theta,  \theta'} > 1-\delta^2/2 ))}{\log M}+\frac{\ln 2}{\log M}\\ 
& \leq \E \frac{Z_{N,\delta}}{Z_{N,0}}+\frac{  \log ( M P^{\otimes 2}_{\Theta}(  \inner{ \theta,  \theta'} > 1-\delta^2/2 ))}{\log M}+\frac{\ln 2}{\log M}.
\end{align*}where in the last inequality we used Jensen's inequality and the fact that the logarithm is concave. Now we send first $N$ to infinity and then $\delta$ to zero and show that the right hand side of the last inequality converges to zero. The third term clearly vanishes. The second first term vanishes by the double limit by using Assumption \ref{assum:full_ov}. The first term vanishes by the first limit since \eqref{assum:no_sol_new}. The proof follows.

\end{proof}

We state and prove here a foklore result in the statistical physics literature called the ``Nishimori" identity, which will be useful in what follows.
\begin{lemma}\label{lem:nishimori}
It always holds that if $\theta'$ is a random variable drawn from the posterior distribution $P_{\theta|Y^n,X^n}$ that
\begin{align*}
\mmse{N}{n}=1-\E \inner{\theta,\theta'}.
\end{align*}

\end{lemma}

\begin{proof}
Bayes' rule implies that the joint distribution of $\inner{\theta',\theta}$ is identical with the distribution of $\inner{\theta',\theta''}$ of two independent random variables drawn from the posterior distribution of $\theta$ given $Y^n,X^n$.
Therefore,
\begin{align*}
\E \inner{\E[\theta|Y^n,X^n],\theta}=\E \inner{\theta',\theta}= \E \inner{\theta',\theta''} = \E \|\E[\theta|Y^n,X^n]\|^2.
\end{align*}The result follows since 
\begin{align*}
\mmse{N}{n}=1+\E \|\E[\theta|Y^n,X^n]\|^2-2\E \inner{\E[\theta|Y^n,X^n],\theta}.
\end{align*}
\end{proof}

\section{Convex analysis}\label{app:conv}

\subsection{Background}
In this work, we use the following two results from convex analysis on the real line.

The first result concerns the left differentiability of a convex function on the interior of its domain.
\begin{theorem}\citep[Proposition I.4.1.1]{HirLem93}\label{thm:conv1}
For any interval $I \subset \mathbb{R},$ convex function $f: I \rightarrow \mathbb{R}$ and $x$ in the interior of $I$, the left derivative of $f$ exists at $x$.
\end{theorem}

The second result establishes that if a sequence of convex function defined on an open interval converges to a convex differentiable function, the pointwise convergence can be generalized to their (left) derivatives.
\begin{theorem}\citep[Proposition I.4.3.4]{HirLem93}\label{thm:conv2}
Fix an open interval $I \subset \mathbb{R},$ and consider a sequence $(f_n)_{n \in \mathbb{N}}: I \rightarrow \mathbb{R}$ of convex functions. Assume that $f_n$ converges pointwise to a differentiable $f: I \rightarrow \mathbb{R}$. Then the left derivatives of $f_n$ converge pointwise to the derivative of $f$, $f'$.
\end{theorem}
\subsection{A Key Proposition}
 Towards employing certain analytic techniques we consider the following function defined on $\RR_{> 0}$, which simply linear interpolates between the values \textcolor{black}{of} the sequence $\kl{P_{n}}{Q_n}/H(\theta), n \in \mathbb{N}$. We establish that the analytic properties of this function express various fundamental statistical properties of the inference setting of interest. Here and throughout this section, $P_n,Q_n$ are defined as in Definitions \ref{dfn:planted}, \ref{dfn:null}.

\begin{definition}\label{dfn:dnstar}
Let $D_{N} : (0,+\infty) \rightarrow [0,+\infty), N \in \mathbb{N}$ be the sequence of functions defined by 
\begin{align}
D_{N} (\beta):=(1-\beta n^*+\floor{\beta n^*})\frac{D(P_{\floor{\beta n^*}}||Q_{\floor{\beta n^*}})}{H(\theta)}+(\beta n^*-\floor{\beta n^*})\frac{D(P_{\floor{\beta n^*}+1}||Q_{\floor{\beta n^*}+1})}{H(\theta)}.
\end{align}

\end{definition}Notice that the normalization of the argument of $D_N$ is appropriately chosen such that $D_N(1)=\kl{P_{n^*}}{Q_{n^*}}/H(\theta).$
\begin{proposition}\label{prop:properties_D}
Consider the sequence of functions $D_{N}, $ per Definition \ref{dfn:dnstar}.
Then under our framework and assumptions described in Section \ref{sec:framework} the following hold.
\begin{itemize}
\item[(a)] For each $N$, $D_{N}$ is a convex, increasing, nonnegative function.
\item[(b)] For all fixed $\beta > 0$,
\begin{align*}
 \limsup_{N}D_{N}(\beta)  = \beta-1+\limsup_{N}\frac{H(\theta|X^{\floor{\beta n^*}},Y^{\floor{\beta n^*}})}{H(\theta)}.
\end{align*}
\item[(c)] For all fixed $\beta > 0$ and for each $N$, the function $D_{N}$ is left differentiable at $\beta$ and the left derivative at $\beta$ satisfies
 \begin{align}\label{eq:deriv}
 (D_{N} )'_{-}(\beta)=1-\frac{H(Y_{\ceil{\beta n^*}}| Y^{\ceil{\beta n^*}-1},X^{\ceil{\beta n^*}})}{H(Y)}+o(1),
\end{align}where the $o(1)$ term tends to zero as $N \rightarrow +\infty$.
\end{itemize}
\end{proposition}

\begin{proof}
We start with part (a).

Since $D_{N}$ is a linear interpolation of the sequence $\frac{\kl{P_n}{Q_n}}{H(\theta)}, n \in \mathbb{N}$ and $H(\theta) > 0$, it suffices to show the same properties for the sequence $\kl{P_n}{Q_n}, n \in \mathbb{N}.$
The nonnegativitiy is obvious.
For a fixed $n \in \mathbb{N}$ we have using the identity \eqref{eq:key_id1} from Proposition \ref{prop:key_identity} that 
\begin{align*}
\kl{P_{n+1}}{Q_{n+1}}-\kl{P_n}{Q_n}=H(Y)-\mi[X^{n+1}]{\theta}{Y^{n+1}}+\mi[X^n]{\theta}{Y^n}
\end{align*}By the chain rule for mutual information, its definition and \textcolor{black}{the independence of the $X_i$'s} we have
\begin{align*}
\mi[X^{n+1}]{\theta}{Y^{n+1}}&=\mi[X^{n+1},Y^{n}]{\theta}{Y_{n+1}}+\mi[X^{n+1}]{\theta}{Y^n}\\
&=\mi[X^{n+1},Y^{n}]{\theta}{Y_{n+1}}+\mi[X^n]{\theta}{Y^n}.
\end{align*}Combining the above and using the definition of the mutual information and the fact that our channel $Y_i=g(X_i,\theta)$ is noiseless, we obtain
\begin{align}
\kl{P_{n+1}}{Q_{n+1}}-\kl{P_n}{Q_n}&=H(Y)-\mi[X^{n+1},Y^{n}]{\theta}{Y_{n+1}} \nonumber \\
&=H(Y)-H(Y_{n+1}|X^{n+1},Y^{n})+H(Y_{n+1}|X^{n+1},Y^{n},\theta) \nonumber \\
&=H(Y)-H(Y_{n+1}|X^{n+1},Y^{n}) \label{DiffDnEnt}\\
&=\mi[]{Y_{n+1}}{X^{n+1},Y^{n}} \label{diffDn}
\end{align}
Now the increasing property of the sequence follows from the fact that the mutual information is non-negative. For the convexity, it suffices to show that the right hand side of \eqref{diffDn} is nondecreasing. Indeed, notice that for each $n$ from the fact that conditioning reduces entropy,
\begin{align*}
\mi[]{Y_{n+1}}{X^{n+1},Y^{n}} &=H(Y_{n+1})-H(Y_{n+1} |X^{n+1},Y^{n})\\
& \geq  H(Y_{n+1})-H(Y_{n+1} |X_2,\ldots,X_{n+1},Y_2,\ldots Y_{n})\\
& = \mi[]{Y_{n+1}}{X_2,\ldots,X_{n+1},Y_2,\ldots Y_{n}}\\
& = \mi[]{Y_{n}}{X_1,\ldots,X_{n},Y_1,\ldots Y_{n-1}}\\
&=\mi[]{Y_{n}}{X^{n},Y^{n-1}}. 
\end{align*}This completes the proof of part (a).

For part (b) notice that from Proposition \ref{prop:key_identity} we have for each fixed $\beta>0$
\begin{align}
\frac{\kl{P_{\floor{\beta n^*}}}{Q_{\floor{\beta n^*}}}}{H(\theta)}&=(1+o(1))\frac{\floor{\beta n^*}}{n^*}-1+\frac{H(\theta|X^{\floor{\beta n^*}},Y^{\floor{\beta n^*}})}{H(\theta)} \nonumber \\
&=\beta-1+\frac{H(\theta|X^{\floor{\beta n^*}},Y^{\floor{\beta n^*}})}{H(\theta)}+o(1), \label{eq:decompDn}
\end{align}since $n^* \rightarrow +\infty$ by Lemma \ref{lem:nstar}.
By~\eqref{DiffDnEnt},
\begin{equation*}
\frac{\kl{P_{\floor{\beta n^*}+1}}{Q_{\floor{\beta n^*}+1}}}{H(\theta)} - \frac{\kl{P_{\floor{\beta n^*}}}{Q_{\floor{\beta n^*}}}}{H(\theta)}\ \leq \frac{H(Y)}{H(\theta)} = o(1)\,,
\end{equation*}
since $H(Y) = O(1)$ and $H(\theta) \to \infty$.

Since $D_{N}(\beta)$ is a convex combination of $\frac{\kl{P_{\floor{\beta n^*}}}{Q_{\floor{\beta n^*}}}}{H(\theta)}$ and $\frac{\kl{P_{\floor{\beta n^*}+1}}{Q_{\floor{\beta n^*}+1}}}{H(\theta)}$, we conclude that 
\begin{align*}
\limsup_{N} D_{N}(\beta) & = 
\beta-1+\limsup_{N} \frac{H(\theta|X^{\floor{\beta n^*}},Y^{\floor{\beta n^*}})}{H(\theta)},
\end{align*}as we wanted.

For part (c), recall that $D_{N}$ is the piecewise linear interpolation of the convex sequence $\frac{\kl{P_{n}}{Q_{n}}}{H(\theta)}$.
By \citep[Proposition I.4.1.1]{HirLem93}, stated also in Theorem \ref{thm:conv1}, $D_N$ possesses a left derivative on the interior of its domain, so that $D_N$ is left-differentiable at $\beta$ for all $\beta > 0$.
Moreover, this left derivative is simply the slope of the segment which connects $(\frac{\ceil{\beta n^*} - 1}{n^*}, \frac{\kl{P_{\ceil{\beta n^*}-1}}{Q_{\ceil{\beta n^*}-1}}}{H(\theta)})$ and $(\frac{\ceil{\beta n^*} }{n^*}, \frac{\kl{P_{\ceil{\beta n^*}}}{Q_{\ceil{\beta n^*}}}}{H(\theta)})$, which equals
\begin{align*}
\frac{\frac{\kl{P_{\ceil{\beta n^*}}}{Q_{\ceil{\beta n^*}}}}{H(\theta)}-\frac{\kl{P_{\ceil{\beta n^*}-1}}{Q_{\ceil{\beta n^*}-1}}}{H(\theta)}}{\frac{\ceil{\beta n^*}}{n^*}-(\frac{\ceil{\beta n^*}-1}{n^*})}&=n^*\frac{\kl{P_{\ceil{\beta n^*}}}{Q_{\ceil{\beta n^*}}}-\kl{P_{\ceil{\beta n^*}-1}}{Q_{\ceil{\beta n^*}-1}}}{H(\theta)}\\
&=(1+o(1))\frac{\kl{P_{\ceil{\beta n^*}}}{Q_{\ceil{\beta n^*}}}-\kl{P_{\ceil{\beta n^*}-1}}{Q_{\ceil{\beta n^*}-1}}}{H(Y)},
\end{align*}where we use the second part of Lemma \ref{lem:nstar} for $n^*$ and the $o(1)$ term tends to zero as $n^*$ tends to infinity. Now using \eqref{DiffDnEnt} we conclude that the slope is 
\begin{align*}
(1+o(1))\left(1-\frac{H(Y_{\ceil{\beta n^*}}| Y^{\ceil{\beta n^*}-1},X^{\ceil{\beta n^*}})}{H(Y)}\right)=1-\frac{H(Y_{\ceil{\beta n^*}}| Y^{\ceil{\beta n^*}-1},X^{\ceil{\beta n^*}})}{H(Y)}+o(1).
\end{align*}
The proof is complete.
\end{proof}

\section{Proof of Theorem \ref{thm:main}: Turning ``all" into ``nothing"}

Recall that from Proposition \ref{prop:ent_post}, condition \eqref{mmse0} implies that  at $(1+\epsilon)n^*$ samples the entropy of the posterior distribution is of smaller order than the entropy of the prior.  Our first result towards proving Theorem \ref{thm:main} establishes two implications of this property of the entropy of the posterior.

\begin{lemma}\label{lem:implic_post}
Suppose that for all $\epsilon>0$ and $n \geq (1+\epsilon)n^*,$ \eqref{eq:all_fake} holds. Then we have,
\begin{itemize}
\item[(1)] (KL closeness) \begin{align}\label{eq:KL_bound}
\lim_{N} \frac{\kl{P_{ n^*}}{Q_{ n^*}}}{H(\theta)}=0,
\end{align} and
\item[(2)] (prediction ``nothing") for any fixed $\epsilon>0$, if $n \leq (1-\epsilon)n^*$, then
\begin{align}\label{eq:fake_nothing} 
\lim_{N}  \frac{H(Y_{n+1}|Y^{n},X^{n+1})}{H(Y)}=1.
\end{align}
\end{itemize}

\end{lemma}
In words,  the sublinear entropy of the posterior implies 1) a ``KL-closeness" between the planted distribution $P_{n^*}$ and the null distribution $Q_{n^*}$, and 2) that the entropy of the observation $Y_{n+1}$ conditioned on knowing the past observations $Y^{n}$ and $X^{n+1}$ is (almost) equal to the unconditional entropy of $Y=Y_{n+1}$. While the first condition is, as already mentioned, hard to interpret (because of the $H(\theta)$ normalization), the second condition has  rigorous implication of the recovery problem of interest, because of the following lemma. We emphasize to the reader that towards establishing this lemma the use of an assumption such as Assumption \ref{assum:r_low_ov} is crucial.

\begin{lemma}\label{lem:nothing}
Suppose that \eqref{eq:fake_nothing} holds. Then for any fixed $\epsilon \in (0,1)$, if $n=n_N \leq (1-\epsilon)n^*$, then \begin{align*}
\lim_N \mmse{N}{n} =1,
\end{align*}i.e. \eqref{mmse1} holds.
\end{lemma}

Notice that combining Proposition \ref{prop:ent_post}, the part (2) of Lemma \ref{lem:implic_post} and Lemma \ref{lem:nothing}, the proof of Theorem \ref{thm:main} follows in a straightforward manner. We proceed by establishing the two lemmas.

\subsection{Proof of Lemma \ref{lem:implic_post}}
\begin{proof} 
We start with establishing \eqref{eq:KL_bound}. Notice that for any $\epsilon>0$, combining Proposition \ref{prop:properties_D} part (b) for $\beta=1+\epsilon$ and the condition \eqref{eq:all_fake}, we have $$\limsup_{N} D_{N}(1+\epsilon) = \epsilon.$$ Using now that $D_{N}$ is increasing from Proposition \ref{prop:properties_D} part (a), we conclude   $$\limsup_{N} D_{N}(1) \leq \epsilon,$$ or as $\epsilon>0$ was arbitrary and $D_N$ is non-negative,
\begin{align}\label{eq:DN0}
\lim_{N} D_{N}(1) =0.
\end{align} The identity \eqref{eq:KL_bound} follows because for $\beta=1$, $\beta n^* \in \mathbb{N}$ and therefore $D_{N}(1)=\frac{\kl{P_{ n^*}}{Q_{ n^*}}}{H(\theta)}.$

We now show that \eqref{eq:KL_bound} implies \eqref{eq:fake_nothing}. Notice that using Proposition \ref{prop:properties_D} part (a), $D_{N}$ is a sequence of increasing, convex and non-negative functions which we restrict to be defined on the compact interval $[0,1]$. Hence combining with \eqref{eq:KL_bound} or the equivalent \eqref{eq:DN0}, we have $$ \lim_N \sup_{\beta \in [0,1]} |D_N(\beta)| = \lim_N D_N(1)=0.$$Therefore $D_{N}$ converges uniformly to the zero function.

Now to establish our result notice that since conditioning reduces entropy it suffices to consider the case where $n=\ceil{ \beta n^*}-1$ for some fixed $\beta \in (0,1)$. Using standard analysis result~\citep[Proposition I.4.3.4]{HirLem93}, stated also in Theorem \ref{thm:conv2}, since the functions $D_N$ are convex and converge uniformly to $0$ in the open interval $(0,1)$ we can conclude the left derivatives of $D_N(\beta)$ converge to the derivative of the zero function as well, i.e. 
 \begin{align*}
 \lim_N (D_{N} )'_{-}(\beta)=0.
\end{align*} Using now \eqref{eq:deriv} from Proposition \ref{prop:properties_D} part (c) for this $\beta$ we conclude the proof.

\end{proof}

\subsection{Proof of Lemma \ref{lem:nothing}}
\begin{proof}Fix some $\epsilon \in (0,1)$ and assume $n \leq (1-\epsilon)n^*$.
Denote the probability distribution $\tilde{P}_{n+1}$ on $(Y^{n+1},X^{n+1})$ where $(Y^{n},X^{n})$ are drawn from $P_{n}$ and $Y_{n+1},X_{n+1}$ are drawn independently from the marginals $P_Y, \cD$ respectively. Notice that $\tilde{P}$ is carefully chosen so that\begin{align*}
\kl{P_{n+1}}{\tilde P_{n+1}} =\E \log \frac{P(Y_{n+1}|Y^{n},X^{n+1})}{P(Y_{n+1})}=H(Y)- H(Y_{n+1}|Y^{n-1},X^{n+1})\,.
\end{align*}Hence using \eqref{eq:fake_nothing},
\begin{align*}
\kl{P_{n+1}}{\tilde P_{n+1}} =o(H(Y))=o(1)\,,
\end{align*}where we used the assumption that $H(Y) \leq \log |\mathcal{Y}|=O(1).$ Using Pinsker's inequality we conclude 
\begin{align*}
\lim_N \tv{P_{n+1}}{\tilde P_{n+1}} =0.
\end{align*}

Now we denote by $\theta'$ a sample from the posterior distribution $P_{n} (\theta|Y^{n},X^{n})$. Using the total variation guarantee we have
\begin{equation*}
P_{n+1}\left\{g(X_{n+1}, \theta') = Y_{n+1}\right\} = \tilde P_{n+1}\left\{g(X_{n+1}, \theta') = Y_{n+1}\right\} + o(1)\,.
\end{equation*}
Under $\tilde P$, because of its definition, we can write $Y_{n+1}$ as $g(X_{n+1}, \theta'')$, where $\theta'' \sim P_\theta$ is independent of everything else. Using Assumption \ref{assum:r_low_ov} we conclude
\begin{equation*}
\E R(\inner{  \theta,\theta'}) = \E R(\inner{ \theta'',\theta'})+ o(1).
\end{equation*}Furthermore, using \eqref{in_prob} and the fact that $\theta''$ is independent from $\theta'$ we conclude that for any $\epsilon>0,$
\begin{equation*}
\limsup_N \E R(\inner{  \theta,\theta'}) \leq R(\epsilon)
\end{equation*} Hence by continuity of $R$ at $0$ we have
\begin{equation*}
 \limsup_N \E R(\inner{  \theta,\theta'}) \leq R(0).
\end{equation*}
Recall that $0$ is the unique minimizer of $R$ on $[0, 1]$ which allows us to conclude
\begin{equation*}
\E |R(\inner{  \theta,\theta'}) - R(0)| = o(1)\,.
\end{equation*} and therefore by Markov's inequality for any $\ep > 0$, $$P( R(\inner{ \theta,\theta'}) > R(0) + \ep ) = o(1).$$As $R$ is strictly increasing we conclude that, for any $\epsilon>0$, $$P(\inner{ \theta,\theta'} >\epsilon)=o(1).$$  Since the integrand is bounded from above we conclude that
$$\limsup_N \E\inner{\theta,\theta'} \leq 0.$$Using now Lemma \ref{lem:nishimori}, we conclude
$$\liminf_N \mmse{N}{n} \geq 1,$$which concludes the proof.

\end{proof}

\section{Proof of Corollary \ref{cor:main}: Establishing the ``all" }

\begin{proof}
We apply Theorem \ref{thm:main}. We fix some $\epsilon>0$ and want to show that for if $n \geq (1+\epsilon)n^*$, \eqref{assum:no_sol} holds, which based on Proposition \ref{prop:postCounting}] implies the desired ``all" condition \eqref{mmse0}. We also assume without loss of generality that $n \leq C n^*$ for an absolute positive constant $C$, since the random variables $Z_{N, \delta}$ are decreasing in the stochastic order as functions of the samples size $n$.  By assumption $p$ is a fixed constant in $(0,1)$ independent of $N$.  Hence since $h(p)$ is a positive constant itself we conclude from the definition of $n^*$ that for all $n=\Theta(n^*)$ it also holds
\begin{align}\label{nlogM} 
n= \Theta(\log M).
\end{align} In particular, $n \rightarrow +\infty$ as $N \rightarrow +\infty$.

Consider $n_1 $ the number of samples where $Y_i=1$ and notice that $n_1$ is distributed as a Binomial distribution $\mathrm{Bin}(n,p).$ We condition on the event that $\mathcal{F}=\{|n_1-np| \leq \sqrt{n} \log \log n \}.$ Standard large deviation theory on the Binomial distribution yields  that since $p \in (0,1)$ and $n \rightarrow +\infty$, it holds $ \lim_N P(\mathcal{F})=1.$

Therefore, by Markov's inequality it suffices to prove that for every $\delta>0$,
\begin{align}
\lim_N \E[ Z_{N,\delta} 1(\mathcal{F})]=0.
\end{align} or equivalently by linearity of expectation and the independence of $Y_i,X_i, i=1,2,\ldots, n$ given $\theta$,

\begin{align*}
\lim_N \E \left\{ 1(\mathcal{F})\sum_{\theta': \|\theta'-\theta\|^2 \geq \delta} P\left( \bigcap_{i=1}^n \{Y_i=g(X_i,\theta')\} \bigg{|} n_1\right)\right\}=0.
\end{align*}Now fix any $\theta'$ with  $\|\theta'-\theta\|^2 \geq \delta$ or equivalently $\rho=\inner{\theta,\theta'}\leq 1-\frac{\delta}{2}$ and some $n_1$ satisfying $\mathcal{F}$.  Using the definitions of $R_i, i=0,1$ from Assumption \ref{assum:r_new} we have that $$P\left(\bigcap_{i=1}^n \{Y_i=g(X_i,\theta')\} \bigg{|}  n_1\right)$$equals
\begin{align}
&\binom{n}{n_1} P\left(g(X,\theta)=g(X,\theta')=1 \right)^{n_1}P\left(g(X,\theta)=g(X,\theta')=0 \right)^{n-n_1} \nonumber\\
=&\binom{n}{n_1} R_1({\rho})^{n_1}R_0({\rho})^{n-n_1} \nonumber\\\
=&\exp \left( n  h(\frac{n_1}{n})+n_1 \log R_1({\rho})+(n-n_1)\log R_0({\rho})+o(n) \right)
\end{align}where $h$ is defined in \eqref{bin_ent}, and  we used the standard application of Stirling's formula $\log \binom{n}{n x}=n h(x)+o(n)$ when $x$ is bounded away from $0$ and $1$. The last expression equals to
\begin{align}
&\exp \left( n  \left(h(\frac{n_1}{n})+\frac{n_1}{n} \log R_1({\rho})+(1-\frac{n_1}{n})\log R_0({\rho}) \right)+o(n) \right) \nonumber\\
=&\exp \left( n  \left(h(p)+p \log R_1({\rho}))+(1-p)\log R_0({\rho})\right)+o(n) \right) \label{F}\\
=&\exp \left( n  \left(p \log \frac{R_1({\rho})}{p}+(1-p)\log \frac{R_0({\rho})}{(1-p)} \right)+o(n) \right) \label{Final},
\end{align} and for \eqref{F} we used the continuity of $h$ and that $n_1/n=p\left(1+O(\log \log n/\sqrt{n})\right)=p(1+o(1)),$ since $n \rightarrow +\infty$.  Importantly, since $p \in (0,1)$ the $o(n)$ term in \eqref{Final} can be taken to hold uniformly over the specific choices of $n_1$ satisfying $\mathcal{F}$.

Using \eqref{Final} it suffices to establish for $G(\rho,p)=p \log \frac{R_1({\rho})}{p}+(1-p)\log \frac{R_0({\rho})}{(1-p)}$ that
\begin{align}\label{eq:goal1}
\lim_N \E_{ \theta } \sum_{ \rho \in \mathcal{R}: \rho \leq 1-\frac{\delta}{2}} |\{\theta' \in \Theta: \inner{\theta,\theta'}=\rho\}| e^{n G(\rho,p)+o(n)}=0,
\end{align}where $\mathcal{R}$ denotes the support of the overlap distribution of two independent samples from the prior $P_{\Theta}$. Now since the prior is uniform over $\Theta$ if $\rho$ is drawn from the law of the inner product between two independent samples from the prior, \eqref{eq:goal1} is equivalent with
\begin{align}
\lim_N \E_{\rho } 1( \rho \leq 1-\frac{\delta}{2})M e^{n G(\rho,p)+o(n)}=0.
\end{align}Now notice that since $n^*=(1+o(1))\frac{\log M}{h(p)}$ by Proposition \ref{lem:nstar} we have
\begin{align} 
M e^{n G(\rho,p)+o(n)}&=\exp\left(\log M+n^*G(\rho,p)+(n-n^*)G(\rho,p)+o\left(n \right) \right) \nonumber\\
&=\exp\left(n^* h(p)+n^*G(\rho,p)+\left(n-n^* \right)G(\rho,p)+o\left(n+\log M \right) \right) \nonumber\\
&=\exp \left( n^*\left(p \log \frac{R_1({\rho})}{p^2}+(1-p)\log \frac{R_0({\rho})}{(1-p)^2}\right)+\left(n-n^* \right)G(\rho,p)+o(\log M)\right), \label{eq:alg} 
\end{align}where we used that $n$ is of order $\log M$, by \eqref{nlogM}.

Now Assumption \ref{assum:r_new} implies that the functions $R_i, i=0,1$ are increasing in $[0,1]$ and Assumption \ref{assum:r_low_ov} that their sum is strictly increasing in $[0,1]$. Furthermore, notice that at full correlation it holds $R_1(1)=p, R_0(1)=1-p$. Hence we conclude that for some $\delta'>0$ the following holds;  for any $\rho \leq 1-\frac{\delta}{2},$  $$\min \{ \log \frac{R_1({\rho})}{p}, \log \frac{R_0({\rho})}{1-p} \} \leq \min \{\log \frac{R_{1}(1)}{p}, \log \frac{R_0({1})}{1-p}\} -\delta'=-\delta'$$ and $$\max \{ \log \frac{R_1({\rho})}{p}, \log \frac{R_0({\rho})}{1-p} \} \leq \max \{\log \frac{R_{1}(1)}{p}, \log \frac{R_0({1})}{1-p}\} =0.$$ Hence, since $p \in (0,1)$, from the definition of $G(\rho,p)$ we conclude that for $\delta''=\delta'\min \{p,1-p\}>0$ it holds that for all $\rho \leq 1-\frac{\delta}{2},$ 
$G(\rho,p) \leq -\delta''.$ Hence since $n \geq (1+\epsilon)n^*$ and $n^*=\Theta(\log M)$ we conclude that for all $\rho \leq 1-\frac{\delta}{2},$ 
\begin{align}\label{G_rho}
(n-n^*)G(\rho,p) \leq -\epsilon \delta'' n^*=-\Omega(\log M).
\end{align}Combining \eqref{eq:alg} with \eqref{G_rho}, and then using $n^*=(1+o(1))\frac{\log M}{h(p)}=\frac{\log M}{h(p)}+o(\log M),$ we conclude
\begin{align*}
&\E_{\rho } 1( \rho \leq 1-\frac{\delta}{2})M e^{n G(\rho,p)+o(n)}\\
& \leq e^{-\Omega(\log M)}  \E_{\rho } \exp \left( n^*(p \log \frac{R_1({\rho})}{p^2}+(1-p)\log \frac{R_0({\rho})}{(1-p)^2})\right) \\
& =e^{-\Omega(\log M)}  \E_{\rho } \exp \left( \frac{\log M}{h(p)}(p \log \frac{R_1({\rho})}{p^2}+(1-p)\log \frac{R_0({\rho})}{(1-p)^2})\right).
\end{align*}
Hence we are left with establishing

\begin{align}\label{eq:goal2}
 \limsup_N \frac{1}{\log M} \log \E_{\rho } \exp \left(  W(\rho,p) \log M \right) =0,
\end{align} for $W(\rho,p) \defeq \frac{1}{h(p)}(p \log \frac{R_1({\rho})}{p^2}+(1-p)\log \frac{R_0({\rho})}{(1-p)^2}).$

To prove it, let us fix a positive integer $k$. We have \begin{align*}
\E_{\rho } \exp \left(  W(\rho,p) \log M \right)& \leq \sum_{\ell = 0}^{k-1} P[\rho \geq \ell/k] \sup_{t \in [\ell/k, (\ell+1)/k)} \exp \left(  W(t,p) \log M \right) \\
& \leq k \cdot \max_{0 \leq \ell < k} \sup_{t \in [\ell/k, (\ell+1)/k]} \exp\left(W(t,p) \log M + \log  P[\rho \geq \ell/k] \right).
\end{align*}

Therefore by using the overlap rate function $r$,
\begin{align*}
 \limsup_N \frac{1}{\log M} \log \E_{\rho } \exp \left(  W(\rho,p) \log M \right) &\leq \max_{0 \leq \ell < k} \sup_{t \in [\ell/k, (\ell+1)/k]} \left(W(t,p)-r\left(\frac{\ell}{k}\right) \right)\\
& \leq   \sup_{t \in [0,1]} \left(W(t,p)-r(t)\right)+\sup_{t,t' \in [0,1]: |t-t'| \leq \frac{1}{k}} |r(t)-r(t')|. 
\end{align*} Sending $k \rightarrow +\infty$ using the uniform continuity of $r$ \textcolor{black}{(implied by e.g. the Heine-Cantor theorem)} we conclude
\begin{align*}
 \limsup_N \frac{1}{\log M} \log \E_{\rho } \exp \left(  W(\rho,p) \log M \right) & \leq   \sup_{t \in [0,1]} \left(W(t,p)-r(t)\right). 
\end{align*} The assumption \eqref{cond:main} completes the proof.

\end{proof}

\section{Applications: the Proofs}

In this section we present the proofs for the three families of models we establish the all-or-nothing phenomenon using our technique. The proof concept remains the same across the different models; we apply Corollary \ref{cor:main} and check that all assumptions apply.

\subsection{Proof of Theorem \ref{thm:BGT}}
\begin{proof}
We apply Corollary \ref{cor:main}. Notice that for any fixed $\theta \in \Theta$ the random variable $\sqrt{k}\inner{X_i,\theta}$ follows a Binomial distribution $\mathrm{Bin}(k,\frac{\nu}{k})$. Therefore 
\begin{align}
p=1-P(g(X,\theta)=0)=1-\left(1-\frac{\nu}{k}\right)^k=1-q.
\end{align}Hence $h(p)=h(q)$ and the critical sample size is indeed  $n^*=\floor*{\frac{\log \binom{N}{k}}{h(q)}},$  and Stirling's formula implies that since $k=o(N)$, $H(\theta)=\log \binom{N}{k}=(1+o(1))k \log \frac{N}{k}$ and therefore it also holds $n^*=(1+o(1))\frac{ k\log \frac{N}{k}}{h(q)}.$

We now check the assumptions of the Corollary.

\paragraph{Assumption \ref{assum:full_ov}} We start with Assumption \ref{assum:full_ov}, which concerns properties of the prior. We use~\citep[Lemma 6]{NilZad20} to conclude that the prior $P_{\Theta}$ admits the overlap rate function $r(t)=t, $ per Definition \ref{dfn:ov_rate}. Now, notice that the first part of Assumption \ref{assum:full_ov} is directly implied by the fact that $r(1+\delta)=1+\delta>1$ for any fixed $\delta>0$. For the second part notice that since the law of the prior is permutation-invariant with respect to the $N$ dimensions, for any fixed $\theta \in \Theta$ and $\theta'$ chosen from the prior, $ \inner{\theta,\theta'}$ is equal in distribution to the law of $ \inner{\theta,\theta'}$ where $\theta,\theta'$ are two independent samples from the prior. Hence using the overlap rate function $r(t)=t$ we have that for $M_N=\binom{N}{k}$ it holds that for any $\epsilon>0$, $$ P\left( \inner{\theta,\theta'}>\epsilon \right) \leq \exp \left(- \left(\epsilon +o(1) \right)\log M_N  \right)=o(1),$$as desired.

\paragraph{Assumptions  \ref{assum:r_low_ov}, \ref{assum:r_new}} For Assumption  \ref{assum:r_low_ov} and Assumption \ref{assum:r_new} we directly compute by elementary combinatorics the functions $R_i({\rho}), i=1,2$ and $R(\rho)$. Recall that $\{g(X,\theta')=1\}$ is the event that the supports of $\theta'$ and $X$ have a non-empty intersection.  Given a $N$-dimensional vector $v \in \mathbb{R}^n$, we denote its support by $S(v):=\{i \in [N]: v_i \not =0 \}$.
First, fix some $\rho \in [0,1]$ and we compute $R_{\rho}(1)$ by considering two arbitrary $\theta,\theta'$ which share $\rho k$ indices in their support. Notice that for the argument to be non-vacuous we assume also that $\rho=\ell/k$ for some $\ell \in \{0,1,2,\ldots, k\}.$ Conditioning on whether $S(X)$ intersects $S(\theta) \cap S(\theta')$, we have
\begin{align*}
R_{1}(\rho)&=P(g(X,\theta)=g(X,\theta')=1)\\
&=P(S(X) \cap S(\theta) \not = \emptyset, S(X) \cap S(\theta') \not = \emptyset)\\
&=\underbrace{1-\left(1-\frac{\nu}{k}\right)^{\ell}}_{ \text{ case } S(X)\cap S(\theta) \cap S(\theta') \not = \emptyset} +\underbrace{\left(1-\frac{\nu}{k}\right)^{\ell} \left(1-\left(1-\frac{\nu}{k}\right)^{k-\ell}\right)^2}_{ \text{ case } S(X)\cap S(\theta) \cap S(\theta')  = \emptyset}\\
&=1-2\left(1-\frac{\nu}{k}\right)^k+\left(1-\frac{\nu}{k}\right)^{2k-\ell}\\
&=1-2q+q^{2-\rho}.
\end{align*}Likewise,
\begin{align*}
R_{0}(\rho)&=P(g(X,\theta)=g(X,\theta')=0)\\
&=P(S(X) \cap S(\theta)  = S(X) \cap S(\theta') = \emptyset)\\
&=\left(1-\frac{\nu}{k}\right)^{2k-\ell}\\
&=q^{2-\rho}
\end{align*}We obtain
\begin{align*}
R(\rho)=1-2q+2q^{2-\rho}.
\end{align*}It can be straightforwardly checked that all three functions are strictly increasing and continuous in $[0,1]$.

\paragraph{ Condition \eqref{cond:main}} Finally, we need to check the condition \eqref{cond:main}. First notice that since $r(t)=t$ we need to show that for all $\rho \in [0,1]$,
\begin{align*}
\rho h(p) \geq  \left(p \log \frac{R_1({\rho})}{p^2}+(1-p) \log \frac{R_0({\rho})}{(1-p)^2} \right),
\end{align*} or using the definition of $h$,
\begin{align}\label{equiv_cond}
0 \geq p \log \frac{R_1({\rho})}{p^{2-\rho}}+(1-p) \log \frac{R_0({\rho})}{(1-p)^{2-\rho}}.
\end{align}Notice that for any $\rho$, $$\frac{R_0({\rho})}{(1-p)^{2-\rho}}=1.$$Therefore it suffices to show that for every $\rho \in [0,1]$,
\begin{align*}
R_1({\rho}) \leq p^{2-\rho}
\end{align*} or equivalently with respect to $q=1-p$,
\begin{align*}
1-2q+q^{2-\rho} \leq (1-q)^{2-\rho}.
\end{align*}
To prove the latter, recall that $q \leq \frac{1}{2}$ and consider the function $f(\rho)=(1-q)^{2-\rho}-q^{2-\rho}.$ Notice that $f(0)=f(1)=1-2q$ and therefore it suffices to prove that $f$ is concave in $[0,1]$. The second derivative of $f$ is \begin{align*}
f''(\rho)&=\log(1-q)^2 (1-q)^{2-\rho}-(\log q)^2 q^{2-\rho}\\
&=(1-q)^{2-\rho} \left(\log(1-q)^2 -(\log q)^2 (\frac{q}{1-q})^{2-\rho}\right)\\
& \leq (1-q)^{2-\rho} \left(\log(1-q)^2 -(\log q)^2 (\frac{q}{1-q})^2\right),
\end{align*}since $q \leq \frac{1}{2}$. Hence, it suffices to show $\log(1-q)^2  \leq (\log q)^2 (\frac{q}{1-q})^2$ or $(1-q) \log(1-q)  \geq q \log q.$ To prove the latter consider the function $g(q)=(1-q) \log(1-q) -q \log q, q \in (0,\frac{1}{2}].$ Notice that $g(0^+)=g(1/2)=0$, and that for each $q \in (0,\frac{1}{2})$ $g''(q)=\frac{2q-1}{q(1-q)} <0.$ Hence, $g(q)$ is concave on the interval $[0, 1/2]$, and $g(q) \geq \min\{g(0^+), g(1/2)\}=0$ as we wanted. The proof is complete.

\end{proof}

\subsection{Proof of Theorem \ref{thm:SBG}  }

\begin{proof}
We apply Corollary \ref{cor:main}. Notice that since any fixed $\theta \in \Theta$ lies on the unit sphere in $\RR^N$ and $X_i \sim N(0,I_N)$, it holds that $\inner{X_i,\theta} \sim N(0,1)$. Hence
$$p=P(g(X_i,\theta)=1)=P( \inner{X_i,\theta} \in A)=\frac{1}{2}.$$ Hence indeed the critical sample size is $n^*=\floor{\frac{ \log \binom{N}{k}}{h(\frac{1}{2})}}=(1+o(1))\floor{k \log _2\frac{N}{k}},$ where we have Stirling's formula and the assumption that $k=o(N)$.

We now check the assumptions of the Corollary.

\paragraph{Assumption \ref{assum:full_ov}} Assumption \ref{assum:full_ov} concerns properties of the prior $P_{\Theta}$ and they are already established in the corresponding part of Theorem \ref{thm:BGT}, since the prior is identical.

\paragraph{Assumptions  \ref{assum:r_low_ov}, \ref{assum:r_new}} For Assumption  \ref{assum:r_low_ov} and Assumption \ref{assum:r_new}, we study the functions $R_i(\rho), i=0,1$ and $R(\rho)$.

Now we compute the functions. Recall that $\{g(X,\theta)=1\}=\{ \inner{X_i,\theta} \in A\}$ and that for any $\theta,\theta' \in \Theta$ with $\inner{\theta,\theta'}=\rho$ the pair $\inner{X_i,\theta},\inner{X_i,\theta'}$ is a bivariate pair of standard Gaussians with correlation $\rho$. Letting $(Z, Z_\rho)$ be such a pair, we therefore have
\begin{align*}
R_{1}(\rho)&=P(g(X,\theta)=g(X,\theta')=1)=P(Z \in A, Z_{\rho} \in A),\\
R_{0}(\rho)&=P(g(X,\theta)=g(X,\theta')=0)=P(Z \not \in A, Z_{\rho} \not \in A),\\
R(\rho)&=P(g(X,\theta)=g(X,\theta'))=P(Z \in A, Z_{\rho} \in A)+P(Z \not \in A, Z_{\rho} \not \in A).
\end{align*} Furthermore, because $A$ is balanced we have for any $\rho \in [0,1]$,
\begin{align*}
R(\rho)& =P(\{Z \in A, Z_\rho \in A\} \cup \{Z \notin A, Z_\rho \notin A\}) \\
& = P(Z \in A, Z_\rho \in A) + P(Z \notin A, Z_\rho \notin A) \\
& = P(Z \in A, Z_\rho \in A) + (1 - P(Z \in A) - P(Z_\rho \in A) + P(Z \in A, Z_\rho \in A)) \\
& = 2 P(Z \in A, Z_\rho \in A)\\
&=2 R_{1}(\rho)\\
&=2R_{0}(\rho).
\end{align*}
The uniform limits are all strictly increasing with respect to $\rho \in [0,1]$ and continuous at $0+,$ by Lemma~\ref{lem:gaussian_increasing} applied to $A$ and $A^C$.

\paragraph{ Condition \eqref{cond:main}} Finally, we need to check the condition \eqref{cond:main}. First notice that similar to Theorem \ref{thm:BGT} the prior admits the overlap rate function $r(t)=t$ and therefore the condition is equivalent with \eqref{equiv_cond}. Notice that \eqref{equiv_cond} simplifies since $p=1/2$ in our case to
\begin{align*}
R_{1}(\rho) \leq 2^{\rho-2},
\end{align*} or
\begin{align}\label{equiv_cond2}
P(Z \in A, Z_\rho \in A)  \leq 2^{\rho-2}.
\end{align}
By Borell's noise stability theorem~\citep{Bor85}, since $P(Z \in A) = 1/2 = P(Z \geq 0)$, we have
\begin{equation*}
P(Z \in A, Z_\rho \in A)  \leq  P(Z \geq 0, Z_\rho \geq 0) = \frac{1}{4} \left(1+\frac{2}{\pi} \arcsin \rho\right)\,,
\end{equation*}
where the equality is by Sheppard's formula~\citep{She99}.

Hence it suffices to show that for all $\rho \in [0,1]$ it holds $2^{\rho} \geq 1+\frac{2}{\pi} \arcsin(\rho).$
We consider the function $g(\rho)=2^{\rho}-1-\frac{2}{\pi} \arcsin(\rho), \rho \in [0,1]$. It suffices to show $g(\rho) \geq 0$ for all $\rho \in [0,1]$.

Now notice $g(0)=g(1)=0$ and $g(\frac{1}{2})=\sqrt{2}-4/3>0.$ We claim that there is no root of $g$ in $(0,1)$ which implies the result by Bolzano's theorem. Arguing by contradiction, if there was a root then the derivative $$g'(\rho)=2^{\rho} \ln 2-\frac{2}{\pi} \frac{1}{\sqrt{1-\rho^2}}$$ would have two roots in $(0,1)$ by Rolle's theorem. Rearranging, this is equivalent to the equation $$\rho \ln 2 +\frac{1}{2} \ln (1-\rho^2)=\ln (\frac{2}{\pi \ln 2})$$ having two roots in $(0,1).$
But the function on the left side is concave and is zero for $\rho = 0$, so it takes each negative value at most once.
Since $\ln (\frac{2}{\pi \ln 2}) <0$ we are done.

\end{proof}

\section{Remaining proofs}\label{sec:remain}

 \begin{proof}[Proof of Proposition \ref{prop:postCounting}]

The equality \eqref{eq:posterior} follows in a straightforward manner from the observation that the posterior of $\theta$ given $Y^n,X^n$ is the uniform measure over the solutions $\theta'$ of equations \eqref{eq:observ2}, and the definition of $Z_{N,\delta}.$

For \eqref{eq:posteriorExp} notice that by Cauchy-Schwarz inequality that if $\theta'$ is drawn from the posterior $P_{\theta|Y^n,X^n}$,
\begin{align*}
\mmse{N}{n} &=\E \|\theta-\E[\theta|Y^n,X^n]\|^2 \\
&= \frac 12  \E \|\theta-\theta'\|^2\\
&= \frac 12 \int_{\delta=0}^2 P(\|\theta-\theta'\|^2 \geq \delta)\\
&= \frac 12 \int_{\delta=0}^2 \E P(\|\theta-\theta'\|^2 \geq \delta|Y^n,X^n)\\
&= \frac 12 \E \int_{\delta=0}^2  P(\|\theta-\theta'\|^2 \geq \delta|Y^n,X^n),
\end{align*}
where in the second line we have used Lemma~\ref{lem:nishimori} and where in the last line we are allowed to exchange the order of integration by Tonelli's theorem as all integrands are non-negative. Notice that finally \eqref{eq:posterior} allows us to conclude \eqref{eq:posteriorExp}.

For the second part, fix some arbitrary $\epsilon \in (0,2]$ and set $A=\{Z_{N,\epsilon}/Z_{N,0}>\epsilon \}$. From \eqref{assum:no_sol} we have $P(A)=o(1)$. Notice that for any $\epsilon'$ with $2 \geq \epsilon' \geq \epsilon$, it holds almost surely $Z_{N,\epsilon'}1(A^c)/Z_{N,0} \leq \epsilon$. Hence, we have using \eqref{eq:posteriorExp},
\begin{align*}
\mmse{N}{n} &\leq \E \int_{\delta=0}^{2}\frac{Z_{N,\delta}}{Z_{N,0}} d\delta\\
& =\E \int_{\delta=0}^{2}\frac{Z_{N,\delta}}{Z_{N,0}} 1(A^c) d\delta+ \E \int_{\delta=0}^{2}\frac{Z_{N,\delta}}{Z_{N,0}} 1(A) d\delta \\
& \leq  \E \int_{\delta=0}^{\epsilon }\frac{Z_{N,\delta}}{Z_{N,0}} d\delta+2\epsilon+P(A)\\
& \leq 3\epsilon+o(1).
\end{align*}Therefore,
\begin{align*}
\limsup_N \mmse{N}{n}  \leq 3\epsilon.
\end{align*} As $\epsilon \in (0,2]$ was arbitrary we conclude \eqref{mmse0}.

The other direction follows in a straightforward manner since for any fixed $\epsilon>0$, 
\begin{align*}
\mmse{N}{n}  \geq  \E \int_{\delta=0}^{\epsilon }\frac{Z_{N,\delta}}{Z_{N,0}} d\delta \geq \epsilon \E \frac{Z_{N,\epsilon}}{Z_{N,0}}.
\end{align*} 
\end{proof}

 \begin{proof}[Proof of Proposition \ref{prop:lb}]
Using \eqref{eq:key_id2} from Proposition \ref{prop:key_identity} we have that since the KL divergence is non-negative,
\begin{align}\label{ineq:KL_pos}
1-\frac{H(\theta|Y^n,X^n)}{H(\theta)} \leq (1+o(1))\frac{n}{n^*}
\end{align}Using now Proposition \ref{prop:ent_post} we conclude that \eqref{eq:all_fake} holds. Combining \eqref{eq:all_fake} with \eqref{ineq:KL_pos} concludes the result.
\end{proof}

$$f(X)$$

$$D_{\infty}(p,q)=\max_S \log p(S)/q(S).$$

\begin{lemma}\label{lem:gaussian_increasing}
Let $Z$ and $Z_\rho$ be a bivariate pair of standard Gaussians with correlation $\rho$.
Then for any Borel set $A \subseteq \RR$ such that $P(Z \in A) \in (0, 1)$, the function
\begin{equation*}
\rho \mapsto P(Z \in A, Z_\rho \in A)
\end{equation*}
is strictly increasing on $[0, 1]$ and continuous on $[0,1)$.
\end{lemma}
\begin{proof}
Write $\gamma$ for the standard Gaussian measure on $\RR$.
We recall~\cite[see, e.g.][Proposition 11.37]{ODo14} that there exists an orthonormal basis $\{h_k\}_{k \geq 0}$ for $L_2(\gamma)$ such that, for any $f \in L_2(\gamma)$,
\begin{equation*}
\E[f(Z) f(Z_\rho)] = \sum_{k \geq 0}  \rho^k \hat f_k^2\,,
\end{equation*}
where the coefficients $\{\hat f_k\}_{k \geq 0}$ are defined by
\begin{equation*}
f = \sum_{k \geq 0} \hat f_k h_k \quad \text{in $L_2(\gamma)$.}
\end{equation*} 
Moreover, $h_0 = 1$, so that if $f$ is not $\gamma$-a.s.~constant, there exists a $k > 0$ for which $\hat f_k \neq 0$.
We obtain that, for any non-constant $f$, the function $\E[f(Z) f(Z_\rho)] = \sum_{k \geq 0}  \rho^k \hat f_k^2$ is strictly increasing on $[0, 1]$. Furthermore, by Parseval's identity $\sum_{k \geq 0}\hat f_k^2 =\E[f^2(Z)] <+\infty.$ Hence then function  $\E[f(Z) f(Z_\rho)] = \sum_{k \geq 0}  \rho^k \hat f_k^2$ is also continuous on $[0,1).$

Applying this result to the non-constant function $f(x) = 1(x \in A) \in L_2(\gamma)$ yields the claim.
\end{proof}
\end{document}